\documentclass{commat}

\usepackage[markup=underlined]{changes}

\usepackage[shortlabels]{enumitem}

\usepackage{listings}
\usepackage{color}

\definecolor{dkgreen}{rgb}{0,0.6,0}

\newtheorem{thm}{Theorem}[section] 

\newtheorem{claim}[thm]{Claim}

\newtheorem{cor}[thm]{Corollary}
\newtheorem{defn}[thm]{Definition}

\newtheorem{prop}[thm]{Proposition}

\theoremstyle{definition}
\newtheorem{rem}[thm]{Remark}
\newtheorem{exmpl}[thm]{Example}

\EditInfo{June 29, 2023}{July 7, 2023}{Ivan Kaygorodov and Mohamed Elhamdadi}

\VOLUME{32}
\YEAR{2024}
\NUMBER{2}
\firstpage{63}
\DOI{https://doi.org/10.46298/cm.11514}

\title{Alternating Roots of Polynomials over Cayley-Dickson Algebras
    }

\author{Adam Chapman and Ilan Levin
    }

\authorinfo[A. Chapman]{
    School of Computer Science, Academic College of Tel-Aviv-Yaffo, Rabenu Yeruham St., P.O.B 8401 Yaffo, 6818211, Israel}{
    adam1chapman@yahoo.com
    }

\authorinfo[
    I. Levin]{
    Department of Mathematics, Bar-Ilan University, Ramat Gan, 55200}{
    ilan7362@gmail.com
    }

\abstract{%
   We introduce the notions of alternating roots of polynomials and alternating polynomials over a~Cayley-Dickson algebra, and prove a~connection between the alternating roots of a~given polynomial and the roots of the corresponding alternating polynomial over the Cayley-Dickson doubling of the algebra. We also include a~detailed Octave code for the computation of alternating roots over Hamilton's quaternions.
    }



\begin{document}

\section{Introduction}

The study of roots of polynomials over fields and more general algebraic objects has been an active area of research for many years.
A detailed description of how to find the roots of a~standard polynomial over $\mathbb{H}$ appeared in~\cite{JanovskaOpfer:2010}, and a~generalization for arbitrary quaternion division algebras appeared in~\cite{ChapmanMachen:2017}.
In~\cite{Chapman:2020}, a~description of the roots of a~standard octonion polynomial was provided, and spherical roots of polynomials over arbitrary Cayley-Dickson algebras were studied in~\cite{ChapmanGutermanVishkautsanZhilina:2023}.

It is difficult, in general, to study the roots of standard polynomials over arbitrary Cayley-Dickson algebras, because the structure of these algebras becomes increasingly more complicated and all the nice properties, like associativity and then alternativity, get lost.
In this paper, we present a~method for finding the roots of a~special kind of polynomials over arbitrary Cayley-Dickson algebras, what we call ``alternating polynomials", and generalize the result for arbitrary polynomials. We also show that in the anisotropic real case, if the degree is odd, then such polynomials always admit roots in the subspace orthogonal to the doubled subalgebra.

\section{Cayley-Dickson Algebras}

Given a~unital $m$-dimensional $F$-algebra $A$ with involution $a \mapsto \overline{a}$, we can construct an algebra $B$ called the Cayley-Dickson doubling of $A$ with parameter $\gamma \in F^{\times}$, in the following way: $B=A \times A$, with multiplication in $B$ given by \[ (a,b) \cdot (c,d) = (ac+\gamma \overline{d} b, da + b \overline{c})\] for any $a,b,c,d \in A$. The involution extends to $\overline{(a,b)} = (\overline{a},-b)$. Let us denote the algebra $B$, which is the algebra resulting of the Cayley-Dickson doubling of $A$, by $CD(A, \gamma)$. Writing $B=A\oplus Av$ by the identification $(a,b) \leftrightarrow a+bv$, we have $v^2 =\gamma$.

The Cayley-Dickson algebras $A_n$ are defined recursively by starting with a~quadratic separable extension $A_1$ of $F$ with the nontrivial automorphism as the involution, and doubling any finite number of times with $\gamma_2,\dots,\gamma_n$. When $\operatorname{char}(F)\neq 2$, \[A_1=F[x : x^2 =\gamma_1]\] is a~Cayley-Dickson doubling of $A_0=F$ with the trivial involution.

For example, taking $F = \mathbb{R}$ and $\gamma_i = -1$ for all $i$, we get the classical Cayley-Dickson construction: $A_1 = \mathbb{C}$, $A_2 = \mathbb{H}$, $A_3 = \mathbb{O}$, $A_4 = \mathbb{S}$, etc.
It should be noted that in each step we lose a~nice property of the previous algebra in the case of the classic Cayley-Dickson construction: $\mathbb{C}$ is not ordered, $\mathbb{H}$ is furthermore not commutative, $\mathbb{O}$ is furthermore not associative (although alternative), and $\mathbb{S}$ is not even alternative (and their norm is not multiplicative, which causes non-trivial zero divisors to appear). A similar loss of properties occurs in the general case as well, which makes $A_n$ extremely hard to work with as $n$ grows (it becomes very hard to work with from $n \geq 4$). Nevertheless, all Cayley-Dickson algebras are flexible, i.e., satisfy $(ab)a=a(ba)$ for any $a$ and $b$, and power-associative~\cite{Schafer:1954}.

Each Cayley-Dickson algebra is equipped with a~quadratic norm form $\operatorname{Norm}: A \to F$ and a~linear trace form $\operatorname{Tr} : A \to F$ given by $\operatorname{Norm}(\lambda) = \overline{\lambda}\lambda$ and $\operatorname{Tr}(\lambda)=\overline{\lambda}+\lambda$. The algebra is quadratic in the sense that every $\lambda \in A$ satisfies \[ \lambda^2 -\operatorname{Tr}(\lambda)\lambda+\operatorname{Norm}(\lambda)=0. \] See~\cite[Chapter 33]{BOI}.

The polynomial ring of a~Cayley-Dickson algebra of $A$ is defined by $A[x] := A \otimes_{F} F[x]$. See~\cite[Section 2.3]{ChapmanGutermanVishkautsanZhilina:2023}.

To specify the type of roots under discussion here, we recall the definition of a~(regular) root as opposed to an alternating root:
\begin{definition}
Let $A$ be a~Cayley-Dickson algebra over $F$ with involution $a \mapsto \overline{a}$ \, and let $f(x) =a_nx^{n}+ \dots +a_1x+a_0 \in A[x]$.
\begin{itemize}
\item $\lambda$ is called a~regular root if $a_n(\lambda^{n})+a_{n-1}(\lambda^{n-1})+ \dots +a_1\lambda+a_0 = 0$.
\item $\lambda$ is called an alternating root if $a_n\dots \overline{\lambda} \lambda \overline{\lambda} \lambda + \dots + a_2 \overline{\lambda} \lambda + a_1 \lambda +a_0 = 0$
\end{itemize}
\end{definition}
\section{Finding alternating roots}
This section is dedicated to the method of finding alternating roots of a~given polynomial, in the case that $A$ is a~division Cayley-Dickson algebra over any field, which guarantees that $A$ has a~multiplicative anisotropic norm, which in turn means that $A$ is of dimension at most 8 over its center.

Let $f(x) =a_nx^{n}+ \dots +a_1x+a_0$. To avoid splitting into cases, suppose $n=2k+1$, but $a_{2k+1}$ might be $0$. In that case, $a_{2k} \neq 0$.

Suppose $\lambda$ is an alternating root of $f$. Let us rewrite the equation $f(x)=0$ as follows:
\[
a_{2k+1}x^{2k+1}+ a_{2k-1}x^{2k-1} + \dots + a_1x = -(a_{2k}x^{2k} + \dots + a_2x^2 + a_0)
\]
Now, plugging in $\lambda$ as an alternating root yields:
\[
(a_{2k+1}(\lambda \overline{\lambda})^{k}+ a_{2k-1}(\lambda \overline{\lambda})^{k-1} + \dots + a_1)\lambda = -(a_{2k}(\lambda \overline{\lambda})^{k} + \dots + a_2(\lambda \overline{\lambda}) + a_0)
\]
But notice that $\operatorname{Norm}(\lambda) = \lambda \overline{\lambda}$, that is, the norm of $\lambda$, which is in the base field $F$ of the Cayley-Dicskon algebra. So, set $N = \operatorname{Norm}(\lambda)$ and we get:
\[
(a_{2k+1}N^{k}+ a_{2k-1}N^{k-1} + \dots + a_1)\lambda = -(a_{2k}N^{k} + \dots + a_2N + a_0)
\]
Now, let us take the norm of both sides. We get:
\[
LHS = \overline{\lambda}(\overline{a_{2k+1}}N^{k}+ \overline{a_{2k-1}}N^{k-1} + \dots + \overline{a_1})(a_{2k+1}N^{k}+ a_{2k-1}N^{k-1} + \dots + a_1)\lambda
\]

\begin{claim}\label{lll}
All the coefficients of \[(\overline{a_{2k+1}}N^{k}+ \overline{a_{2k-1}}N^{k-1} + \dots + \overline{a_1})(a_{2k+1}N^{k}+ a_{2k-1}N^{k-1} + \dots + a_1)\] as a~polynomial in the variable $N$ are central, i.e., in $Z(A)$.
\end{claim}
\begin{proof}
They are all sums of norms $\operatorname{Norm}(a_r)=\overline{a_r}a_r$ and traces $\operatorname{Tr}(\overline{a_r}a_s)=\overline{a_r}a_s+\overline{a_s}a_r$ like in~\cite{JanovskaOpfer:2010}.
\end{proof}

So, overall, we get that:
\[
LHS = N(\overline{a_{2k+1}}N^{k}+ \overline{a_{2k-1}}N^{k-1} + \dots + \overline{a_1})(a_{2k+1}N^{k}+ a_{2k-1}N^{k-1} + \dots + a_1)
\]
which is central.

And, for the right hand side:
\[
RHS = (\overline{a_{2k}}N^{k} + \dots + \overline{a_2}N + \overline{a_0})(a_{2k}N^{k} + \dots + a_2N + a_0).
\]

\begin{claim}
All the coefficients of \[ -(\overline{a_{2k}}N^{k} + \dots + \overline{a_2}N + \overline{a_0})(a_{2k}N^{k} + \dots + a_2N + a_0)\] as a~polynomial in the variable $N$ are central in $A$.
\end{claim}
\begin{proof}
Similar to the proof of claim~\ref{lll}.
\end{proof}

So, overall, we get that $p(N)=RHS-LHS$ is a~polynomial in $N$ with central coefficients.
\begin{prop}
The intersection between the set of central roots of $p(N)$ and the image of the norm form $\operatorname{Norm} : A \rightarrow F$ is the set of norms of alternating roots of $f(x)$.
For each such norm $N_0$ in this intersection, either all the elements in $A$ of this norm are alternating roots of $f(x)$, which happens when $LHS|_{N=N_0}=0$, or the unique $\lambda$ satisfying \[ (a_{2k+1}N_0^{k}+ a_{2k-1}N_0^{k-1} + \dots + a_1)\lambda = -(a_{2k}N_0^{k} + \dots + a_2N_0 + a_0)\] is the one and only alternating root of $f(x)$ of norm $N_0$.
\end{prop}
\begin{proof}
Most of the argument follows immediately from the discussion above.
It is left to explain the case of $LHS|_{N=N_0}=0$. In this case, also $RHS|_{N=N_0}=0$, because $N_0$ is a~root of $p(N)=RHS-LHS$.
In this case, for each $\lambda$ with norm $N_0$, the norm of $a_{2k}(\lambda \overline{\lambda})^{k} + \dots + a_2(\lambda \overline{\lambda}) + a_0$ is zero, which implies that \[a_{2k}(\lambda \overline{\lambda})^{k} + \dots + a_2(\lambda \overline{\lambda}) + a_0=0\] because the norm form is anisotropic, and similarly \[(a_{2k+1}N^{k}+ a_{2k-1}N^{k-1} + \dots + a_1)\lambda =0\] for the same reason, and thus the necessary and sufficient condition for being an alternating root is satisfied automatically for any $\lambda$ of norm $N_0$.
\end{proof}

\begin{exmpl}\label{oneex}
Consider $f(x)=x^2 +ix-1-ij \in \mathbb{H}[x]$.
In this case, $RHS=N^2 -2N+2$ and $LHS=N$, and thus \[ p(N)=N^2 -3N+2=(N-2)(N-1). \] Therefore, the norms of alternating roots are either $N_1=1$ or $N_2=2$. In the case of $N_1=1$, we obtain $i\lambda_1=ij$, and thus $\lambda_1=j$, and in the case of $N_2=2$, we obtain $i\lambda_2=-1+ij$, and thus $\lambda_2=i+j$.
\end{exmpl}

\begin{rem}
A polynomial may not have any alternating roots, for example, $f(x)=x^2 +1$ over Hamilton's quaternions $\mathbb{H}$. If it had an alternating root $\lambda$, then $\overline{\lambda}\lambda+1=0$, which means $\operatorname{Norm}(\lambda)+1=0$, but the norm is always non-negative.
\end{rem}
\begin{rem}\label{root exists}
If $A=A_n$ is the Cayley-Dickson algebra constructed over the real numbers $A_0=\mathbb{R}$ by repeatedly choosing $\gamma=-1$ and \[f(x) = a_{2k+1}x^{2k+1} + \dots + a_1x+a_0\in A[x]\] is a~polynomial of odd degree over $A$ with $a_{2k+1} \neq 0$, then $f(x)$ has an alternating root in $A$.
\end{rem}
\begin{proof}
If $a_0=0$, then $0$ is an alternating root, so assume $a_0 \neq 0$.
Notice that the final polynomial $p(N)$, in the case $n=2k+1$, is of degree $2k+1$, the coefficient of $N^{2k+1}$ is $\overline{a_{2k+1}}a_{2k+1}$, which is positive, and the coefficient of $1$ is $-\overline{a_{0}}a_{0}$, which is negative. Now, notice that $p(0)= -\overline{a_{0}}a_{0} < 0$, and $\lim_{N\to\infty} p(N) = \infty$, so by the Intermediate Value Theorem, there is $N_0 \in \mathbb{R}^{\geq 0}$ such that $p(N_0)=0$. Hence, there is an alternating root.
\end{proof}

\subsection{Octave Code}
We include here an octave code for computing the alternating roots of quaternion polynomials over $\mathbb{H}$. In writing the original code, we were assisted by Ido Simon, the first author's research assistant. The final version follows the one suggested by the referee.

\begin{lstlisting}
function [A, b] = AltRoots(pol)
% The function receives the coefficients of a~quaternion polynomial
% ordered from degree zero to the highest degree
% and returns the vector A of alternating roots
% together with a~vector b stressing whether the solution is isolated
% when b = 0 or spherical when b = 1.
  L = length(pol.w);
  if mod(length(pol), 2)
    pol = [pol, 0];
    L = L + 1;
  endif
  eq_even = pol(1:2:L-1);
  eq_odd = pol(2:2:L);
  N_even = Norm(eq_even);
  N_odd = Norm(eq_odd);
  eq = [0, N_odd] - [N_even, 0];
  R = roots(flip(eq));
  A = [];
  b = [];
  for i = 1:length(R)
    if isreal(R(i)) && (R(i) >= 0)
      r1 = -quPolyval(eq_odd, R(i));
      r0 = quPolyval(eq_even, R(i));
      if (isreal(r1)) && (r1.w == 0)
        A = [A, quaternion(sqrt(R(i)),0,0,0)];
        b = [b, 1];
      else
        A = [A, inv(r1) * r0];
        b = [b, 0];
      endif
    endif
  endfor
endfunction

function Pol = Norm(pol)
  L = length(pol.w);
  Pol = zeros(1, 2 * L - 1);
  for i = 1:L
    for j = 1:L
      Pol(i + j - 1) += (pol(i) * conj(pol(j))).w;
    endfor
  endfor
endfunction

function val = quPolyval(pol, q)
  L = length(pol);
  val = 0;
  power = 1;
  for i = 1:L
    val += pol(i) * power;
    power *= q;
  endfor
endfunction
\end{lstlisting}

\begin{exmpl}
Running
\begin{lstlisting}
pkg load quaternion
pol=[quaternion(-1,0,0,-1),quaternion(0,1,0,0),quaternion(1,0,0,0)];
[A, b]=AltRoots(pol)
\end{lstlisting}
returns the two isolated roots described in Example~\ref{oneex}.
\end{exmpl}

\section{Correspondence theorem}
Throughout this section, let $A$ be an arbitrary Cayley-Dickson $F$-algebra, $B$ its Cayley-Dickson doubling, i.e., $B = CD(A,\gamma)$, for a~fixed $\gamma \in F^{\times}$. Moreover, let $f_{alt}(\lambda)$ denote the result of plugging in $\lambda$ alternatively in $f(x)$.

\begin{defn}\label{altf}
Let $f(x) =a_nx^{n}+ \dots +a_1x+a_0 \in A[x]$. We define $\widetilde{f}(x) \in B[x]$, called the alternating friend of $f(x)$, to be the following polynomial:
\[
\widetilde{f}(x) = b_nx^{n}\dots + b_2x^{2}+b_1x+b_0
\]
where $b_{2i}= \frac{1}{\gamma^i} \overline{a_{2i}}$ and $b_{2i+1}= \frac{1}{\gamma^{i+1}} \overline{a_{2i+1}}v$, when $0 \leq 2i,2i+1 \leq n$.
\end{defn}

\begin{theorem}\label{altthm}
For $\lambda \in A$, $\overline{f_{alt}(\lambda)}=\widetilde{f}(\lambda v)$.
\end{theorem}
\begin{proof}
Let $N$ denote the norm of $\lambda$.
It suffices to show that $b_{2k+1} (\lambda v)^{2k+1} = \overline{a_{2k+1} N^{k}\lambda}$, and $b_{2k}(\lambda v)^{2k} = \overline{a_{2k}N^{k}}$ (as summing over all of them gives yields the desired result).
Let us note the following three facts: both $N$ and $\gamma$ are central,
$(\lambda \nu)^2 = \gamma \overline{\lambda} \lambda = \gamma N$ and
$(\overline{a_{2k+1}} \nu) (\lambda \nu) = \gamma \overline{\lambda}
\overline{a_{2k+1}}$. Then
\begin{itemize}
\item $
\begin{aligned}
&& b_{2k+1} (\lambda \nu)^{2k+1} = b_{2k+1} ((\lambda \nu) ((\lambda
\nu)^2)^k) = b_{2k+1} ((\lambda \nu) (\gamma N)^k)= \\ 
&& \frac{1}{\gamma^{k+1}} \gamma^k N^k (\overline{a_{2k+1}} \nu)(\lambda
\nu) = N^k \overline{\lambda} \overline{a_{2k+1}} =
\overline{a_{2k+1}N^k \lambda};   
\end{aligned}
$ 
\item $
    b_{2k} (\lambda \nu)^{2k} = b_{2k} ((\lambda \nu)^2)^k = b_{2k}
(\gamma N)^k = N^k \overline{a_{2k}} = \overline{a_{2k} N^k}.
$
\end{itemize}
\end{proof}

\begin{cor}
An element $\lambda$ in $A$ is an alternating root of $f(x) \in A[x]$ if and only if $\lambda v$ is a~regular root of $f(x)$'s alternating friend, i.e., $ \widetilde{f(x)} \in B[x]$.
\end{cor}
\begin{proof}
$\overline{f_{alt}(\lambda)} = 0 \iff f_{alt}(\lambda)= 0 \iff \widetilde{f}(\lambda v) = 0$. The last equality is given by the theorem.
\end{proof}

\begin{cor}
Let $A$ be a~division Cayley-Dickson $\mathbb{R}$-Algebra. For all polynomials of odd degree $f(x) = a_{2k+1}x^{2k+1} + \dots + a_1x+a_0$ in $A$, its alternating friend $\widetilde{f}(x)$ has a~regular root in $Av$.
\end{cor}
\begin{proof}
By Remark~\ref{root exists}, $f$ has an alternating root $\lambda$ in $A$, and by the theorem, $\lambda v$ is a~regular root of $f$'s alternating friend, $\widetilde{f}(x)$.
\end{proof}

\section{Generalized result}

\begin{rem}\label{miracle}
Let $f(x) \in A[x]$, and let $\widetilde{f}(x)$ be its alternating friend. Then $\widetilde{f}(Av) \subseteq A$, i.e., the image of $Av$ under $\widetilde{f}$ is in $A$.
\end{rem}
\begin{proof}
By Theorem~\ref{altthm}, for any $\lambda \in A$, $\widetilde{f}(\lambda v)=\overline{f_{alt}(\lambda)} \in A$.
\end{proof}

Let $A$ be a~Cayley-Dickson algebra over $F$, and $B=A \oplus Av$ its doubling with $v^2 =\gamma$. Let $h(x) \in B[x]$ be a~polynomial, and write $h(x)=\widetilde{f}(x)+\widetilde{g}(x)v$ for appropriate polynomials $f(x),g(x)\in A[x]$. (Their existence is an easy straight-forward computation.)

\begin{thm}
For every $\lambda \in A$, $\lambda v$ is a~root of $h(x)$ if and only if $\lambda$ is an alternating root of $f(x)$ and $\overline{\lambda}$ is an alternating root of $g(x)$.
\end{thm}

\begin{proof}
It is enough to prove that plugging in $\lambda v$ in the polynomial $G(x)=\widetilde{g}(x)v$ gives the same result as the product of $\widetilde{g}(\overline{\lambda} v)$ and $v$, because $h(\lambda v)=\widetilde{f}(\lambda v)+G(\lambda v)$, and $\widetilde{f}(\lambda v)$ is in $A$ while $G(\lambda v)$ is in $Av$, and so $h(\lambda v) = 0$ if and only if both $\widetilde{f}(\lambda v)=0$ and $G(\lambda v)=0$.
Write $N$ for the norm of $\lambda$. Now, for each term $\widetilde{g}_n(x)=b_n x^n$ in $\widetilde{g}(x)$, the corresponding term in $G(x)$ is $G_n(x)=b_n v x^n$, and so, when $n=2k$, we have $b_n \in A$, 
\[
G_n(\lambda v)=(b_n v)(\lambda v)^{2k}=N^k \gamma^k b_n v ~\text{ and }~ \widetilde{g}_n(\overline{\lambda} v) \cdot v=(b_n (\overline{\lambda} v)^{2k})v=N^k \gamma^k b_n v,
\] 
and when $n=2k+1$, we have $b_n=a_n v\in Av$, 
\[ G_n(\lambda v)=((a_n v) v)(\lambda v)^{2k+1}=N^k \gamma^{k+1} (a_n)(\lambda v)=N^k \gamma^{k+1} (\lambda a_n)v 
\]
and 
\[ \widetilde{g}_n(\overline{\lambda} v) \cdot v =((a_n v) (\overline{\lambda} v)^{2k+1})\cdot v=(N^k \gamma^k (a_n v) (\overline{\lambda} v))v=(N^k \gamma^{k+1} \lambda a_n)v.
\]
It follows that $\lambda v$ is a~root of $h(x)$ if and only if $\lambda v$ is a~root of $\widetilde{f}(x)$ and $\overline{\lambda} v$ is a~root of $\widetilde{g}(x)$. In turn, $\lambda v$ is a~root of $\widetilde{f}(x)$ if and only if $\lambda$ is an alternating root of $f(x)$, and $\overline{\lambda}v$ is a~root of $\widetilde{g}(x)$ if and only if $\overline{\lambda}$ is an alternating root of $g(x)$.
\end{proof}

\section*{Acknowledgements}
The authors are very much indebted to the referee whose detailed referee report improved the quality of the paper considerably. The authors also thank Uzi Vishne for the helpful comments on the manuscript.
The original Octave code in Section 3 was written by Ido Simon, the first author's research assistant, whose hiring is supported by the internal research grant of the Acadmic College of Tel-Aviv-Yaffo.



\end{document}